\documentclass[11pt]{article}
\usepackage{amsmath,graphicx,enumerate,textcomp}
\usepackage[abbrev]{amsrefs}
\usepackage[a4paper,top=1.5in, bottom=1.5in, left=1in, right=1in, headsep=0.in, centering]{geometry}
\usepackage{amssymb,amsfonts,amsthm}
\usepackage{epsfig}
\numberwithin{equation}{section}

\newcommand{\eps}{\varepsilon}

\newtheorem{theorem}{Theorem}[section]
\newtheorem*{theoremA}{Theorem A}
\newtheorem{lemma}{Lemma}[section]
\newtheorem{proposition}{Proposition}[section]
\theoremstyle{definition}
\newtheorem{remark}{Remark}[section]

\newcommand{\N}{{\mathbb N}}
\newcommand{\R}{{\mathbb R}}
\newcommand{\K}{\mathbf{K}}

\title{Gelfand type problem for two phase porous media}

\author{Peter  V. Gordon
\thanks{
Department of  Mathematics, The University of Akron, Akron, Ohio 44325, USA.  E-mail: {\tt pgordon@uakron.edu}}
\and Vitaly Moroz
\thanks{
Department of Mathematics, Swansea University, Singleton Park, Swansea, SA2 8PP,
Wales, United Kingdom. E-mail: {\tt V.Moroz@swansea.ac.uk}
}}

\begin{document}
\date{}
\maketitle

\begin{abstract}
We consider a generalization of the Gelfand problem arising in  Frank-Kamenetskii theory of thermal explosion.
This generalization is a natural extension of the Gelfand problem to  two phase materials, where, in
contrast to the classical Gelfand problem which utilizes single temperature approach,
the state of the system is described by two different temperatures.
We show that similar to the classical Gelfand
problem the thermal explosion occurs exclusively due to the absence of
stationary temperature distribution. We also show that the presence of
inter-phase heat exchange delays a thermal explosion.
Moreover, we prove that in the limit of infinite heat exchange between phases
the problem of thermal explosion in two phase porous media reduces to the
classical Gelfand problem with renormalized constants.
\smallskip

%
\end{abstract}

\setcounter{equation}{0}

\section{Introduction}
Superlinear parabolic equations and systems of such equations  serve as mathematical models of
 many  nonlinear phenomena arising in natural sciences. It is well known that such models
 may often produce  solutions that do not exist globally  in time due to formation singularities.
In particular, there are solutions which become infinite   either somewhere or everywhere in the spatial domain  in a finite time.
Formation of such  singularities is commonly referred to as blow up  and has attracted considerable attention of scientists and
engineers over past decades \cites{Matano,Galaktionov,Souplet}.
The classical problem in a theory of blow up for nonlinear parabolic equations, which is  widely known in
mathematical literature as a Gelfand problem, reads
\begin{eqnarray}\label{eq:G}
\left\{\begin{array}{ll}
W_t - \Delta W=\Lambda g(W) &  \mbox{in} \quad (0,T)\times \Omega, \\
W=0& \mbox{on} \quad \partial \Omega,\\
W(0,\cdot)=0& \mbox{in}\quad \Omega,
\end{array}\right.
\end{eqnarray}
where $\Omega\subset \mathbb{R}^N$ is a smooth bounded domain, $g: \R\to(0,\infty)$
is a $C^1$ convex non-decreasing function satisfying
\begin{eqnarray}\label{eq:g}
\int_{x_0}^{\infty}\frac{ds}{g(s)}<\infty \quad \mbox{for some} \quad x_0\ge 0,
\end{eqnarray}
and $\Lambda>0$ is a parameter.
This problem was originally introduced in a context of thermo-diffusive combustion
 as a model of thermal explosion,  the spontaneous development of rapid rates of heat release by chemical reactions in combustible mixtures and materials being initially in a non-reactive state \cites{ZBLM,Shteinberg}.
The model \eqref{eq:G}  describes an evolution of initially uniform  temperature field $W$
which diffuses in space, increases in a bulk due to the heat release described by a reaction term $\Lambda g$ and is fixed on the boundary (cold boundary). The model \eqref{eq:G} was derived by Frank-Kamenetskii  \cite{FK}
as a short time asymptotic of
a  standard thermo-diffusive
model and  describes an initial stage of self-ignition of combustible mixture.

Depending on the parameters of this problem the solutions of  \eqref{eq:G}  either   blow up or exist globally.
In a context of  combustion  the first case corresponds to successful initiation of combustion process
whereas the second one corresponds to the ignition failure. Basic physical reasoning discussed in \cites{FK,ZBLM} and
formal (intermediate asysmptotics) arguments of Barenblatt presented in \cite{Gelfand} suggest that blow up in model
\eqref{eq:G} occurs exclusively due to the absence of stationary solutions for this problem.
That is the absence of stationary temperature distribution $w$ that solves, in a weak sense, the following time independent problem
\begin{eqnarray}\label{eq:GS}
\left\{\begin{array}{rcll}
- \Delta w&=&\Lambda g(w),\quad w>0 &\mbox{in} \quad \Omega, \\
w&=&0& \mbox{on} \quad \partial \Omega.
\end{array}\right.
\end{eqnarray}
These formal arguments of \cite{Gelfand} were made rigorous in \cite{Brezis-96}.
The following theorem summarizes the main results regarding solutions of problems \eqref{eq:G} and \eqref{eq:GS}, see \cite{Brezis-96},  \cite{Cazenave}*{Theorem 3.4.1} and further references therein.

\begin{theoremA}

\smallskip

Parabolic problem \eqref{eq:G} has a classical global solution if and only if stationary problem \eqref{eq:GS} has a weak solution.

There exists $0<\Lambda^*<\infty$ such that:

\smallskip
$i)$ for $\Lambda>\Lambda^*$ problem \eqref{eq:GS} has no weak solutions,

$ii)$ for $0<\Lambda<\Lambda^*$ problem \eqref{eq:GS} has a minimal classical solution $w_\Lambda$,

$iii)$ $w_\Lambda(x)$ is a monotone increasing functions of $\Lambda$, and for $\Lambda=\Lambda^*$ problem \eqref{eq:GS} has a weak solution $w^*$ defined by
\begin{eqnarray}
w^*(x):=\lim_{\Lambda\to\Lambda^*}w_\Lambda(x).
\end{eqnarray}
\end{theoremA}

The statement of Theorem A, from a physical perspective, has a very clear interpretation.
Indeed, the parameter $\Lambda$ can be understood as a scaling factor
that reflects the size of the domain, which increases as $\Lambda$ increases.
Thus, in relatively small domains the cold boundary suppresses intensive chemical
reaction in the bulk which leads to a stationary temperature distribution, whereas when
the size of the domain exceeds some critical value corresponding to $\Lambda^*$
the cooling on the boundary becomes insufficient to prevent chemical reaction inside the domain
$\Omega$,
which leads to thermal explosion.

The classical model \eqref{eq:G}  asserts that the process of combustion can be described
using a unified single temperature approach. This assumption, which has  rather wide range of validity,
however, is not applicable in certain situations.  For example in  combustion
of porous  materials the difference of temperatures
of gaseous and condensed phases  can be substantial, which changes  a combustion
process \cite{Shteinberg}.  As a result the model describing self ignition of porous media has to be
appropriately modified. Let us note that explosion in two phase materials has many technological applications ranging from ignition of metal nano-powders and solid rocket propellants to issues of safe storage of nuclear waste and industrial raw garbage  \cites{Dreizin,Shteinberg}.

In order to describe explosion in  two phase materials one may  adopt an
approach of Frank-Kamenetskii and make a standard reduction of governing
equations describing combustion of two phase porous materials. The conventional
system of equation for the dynamics of two phase material are well known and we
refer the reader to \cite{Margolis} for the details.  Partial linearization of these equations
incorporating Frank Kamenetskii transform \cites{FK,ZBLM} lead to a following system
\begin{eqnarray}\label{eq:PT}
\left\{\begin{array}{rcll}
U_t - \Delta U&=&\lambda g(U)+\nu(V-U),& \\
\alpha V_t- d\Delta V&=&\nu (U-V) &  \mbox{in} \quad (0,T)\times \Omega, \\
U=V&=&0& \mbox{on} \quad \partial \Omega,\\
U(0,\cdot)=V(0,\cdot)&=&0& \mbox{in}\quad \Omega,
\end{array}\right.
\end{eqnarray}
here $U(t,x)$ and $V(t,x)$ are  appropriately normalized temperatures of condensed (solid)  and gaseous phases respectively and $d>0$ is a ratio of effective gaseous and thermal diffusivity, $\nu>0$ is  inter-phase heat transfer coefficient and $\alpha>0$ is a parameter which depends on porosity and ratios of specific heats
of the solid and gaseous phases.
 It is important to note that the model \eqref{eq:PT} is formally identical to the one describing
 formation of hot spots in transistors. In this case variables $U$ and $V$ can de interpreted
 as temperatures of electron gas and of the lattice respectively see e.g. \cite{Osipov}.

 As one may expect the behavior of solutions for the problem \eqref{eq:PT}
 depends crucially on existence of stationary solutions for the time independent problem
\begin{eqnarray}\label{eq:P}
\left\{\begin{array}{rcll}
- \Delta u&=&\lambda g(u)+\nu(v-u),& \\
- d\Delta v&=&\nu (u-v) &  \mbox{in} \quad \Omega, \\
u,v&>& 0 &  \mbox{in} \quad \Omega, \\
u=v&=&0& \mbox{on} \quad \partial \Omega.
\end{array}\right.
\end{eqnarray}

The goal of this paper is to study dynamics of solutions for the problem \eqref{eq:PT} and
and its stationary states described by \eqref{eq:P}.  There are two main results of this paper.
Our first result states that similar to the classical Gelfand problem blow up in
system \eqref{eq:PT} is fully determined by solutions of problem \eqref{eq:P}.
Namely the following holds.
\begin{theorem}\label{t:1}
If elliptic problem  \eqref{eq:P} has a classical solution, then parabolic problem \eqref{eq:PT} has a global
classical solution. If parabolic problem \eqref{eq:PT} has a global classical solution, then elliptic problem
has a weak solution.
Moreover, the global classical solution of \eqref{eq:PT}
converges in $L^1$-norm to a minimal weak solution of \eqref{eq:P} as $t\to\infty$.
\end{theorem}

The precise definition of a minimal classical and weak solution of stationary problem \eqref{eq:P} will be given later in Section \ref{Elliptic}. Here we note only that every classical solution is also a weak solution. On the other hand weak solutions may have singularities.
\smallskip

In a view of this result  a detailed information on stationary solutions is needed.
This is given by the following theorem.

\begin{theorem} \label{t:2}
Let $d>0$.
Then for every $\nu>0$ there exists $0<\lambda_\nu^*<\infty$ such that:
\smallskip

$i)$ for $\lambda>\lambda_\nu^*$ system \eqref{eq:P} has no classical solutions,

$ii)$ for $0<\lambda<\lambda_\nu^*$ system \eqref{eq:P} has a minimal classical solution $(u_{\lambda,\nu},v_{\lambda,\nu})$,

$iii)$ for $\nu>0$ both $u_{\lambda,\nu}(x)$ and $v_{\lambda,\nu}(x)$ are monotone increasing functions of $\lambda$  for every $x\in\Omega$, and for $\lambda=\lambda_\nu^*$ system \eqref{eq:P} has a weak solution $(u_\nu^*,v_\nu^*)$ defined by
\begin{eqnarray}
u_\nu^*(x):=\lim_{\lambda\to\lambda_\nu^*}u_{\lambda,\nu}(x),\qquad v_\nu^*(x):=\lim_{\lambda\to\lambda_\nu^*}v_{\lambda,\nu}(x)\qquad(x\in\Omega).
\end{eqnarray}

$iv)$ $\lambda_\nu^*\ge\Lambda^*$ and $\lambda_\nu^*=\lambda^*(\nu)$ is a nondecreasing function of $\nu>0$ having the following properties
\begin{eqnarray}
\lim_{\nu\to 0}\lambda^*(\nu)=\Lambda^*,\qquad \lim_{\nu\to\infty}\lambda^*(\nu)= \Lambda^* (1+d),
\end{eqnarray}
where $\Lambda^*$ is the critical value of the classical Gelfand problem \eqref{eq:GS}.

$v)$
for $\lambda<\lambda_\nu^*$,
$u_{\lambda,\nu}(x)$
is a non-increasing function of $\nu$ for every $x\in \Omega$.
For $\lambda<\Lambda^\ast$ and $\nu\to 0$ solution $(u_{\lambda,\nu},v_{\lambda,\nu})$
converges uniformly to $(u_0,0)$, where $u_0$ is the minimal solution of
\begin{eqnarray}\label{eq:P0}
\left\{\begin{array}{rclll}
- \Delta u_0&=&\lambda g(u_0),& u_0>0 &\mbox{in} \quad \Omega, \\
u_0&=&0 & &\mbox{on} \quad \partial \Omega.\\
\end{array}\right.
\end{eqnarray}
For $\lambda<\Lambda^\ast(1+d)$ and $\nu\to \infty$ solution $(u_{\lambda,\nu},v_{\lambda,\nu})$
converges uniformly to $(u_{\infty}, u_{\infty})$, where $u_\infty$ is the minimal solution of
\begin{eqnarray}\label{eq:Pinf}
\left\{\begin{array}{rclll}
- \Delta u_{\infty}&=&\frac{\lambda}{1+d}\, g(u_{\infty}), &
u_{\infty}>0 &  \mbox{in} \quad \Omega, \smallskip\\
u_{\infty}&=&0 & &\mbox{on} \quad \partial \Omega.\\
\end{array}\right.
\end{eqnarray}
\end{theorem}

\begin{remark}
The limit weak solution $(u_\nu^*,v_\nu^*)$, constructed in $(iii)$
might be either classical or singular.
In the proof of part $(i)$ of Theorem \ref{t:2} we show that
there exists $\lambda_\nu^{**}>\lambda_\nu^*$ such that
system \eqref{eq:P} has no weak solutions for $\lambda>\lambda_\nu^{**}$.
In the case of the single equation \eqref{eq:GS} it is known that actually $\lambda^{**}=\lambda^*$,
see $(i)$ and $(ii)$ of Theorem A.
One may expect that a similar result holds for the system considered in this paper.
However, the proof  of this fact is very delicate even in the case of a single equation,
see \cite[Theorem 3]{Brezis-96} and further discussion in \cite{Brezis-97}.
We also want to point out that in the case of single equation \eqref{eq:GS}
weak solutions corresponding to $\lambda=\lambda^*$ in most of the
cases relevant to applications are in fact classical.
\end{remark}

The Theorem \ref{t:2}  proves that solutions for the problem
\eqref{eq:P}  behave similarly to solutions of the classical Gelfand problem
and that the presence of the heat exchange increases the value of critical
parameter $\lambda^*$. These results are quite in line with the physical
intuition behind this problem. What we found  rather surprising is
the limiting behavior of solutions of \eqref{eq:P} when $\nu\to \infty$.
Indeed, it is quite remarkable that the  substantial heat exchange between
the two phases reduces  problem \eqref{eq:P}  to the classical Gelfand problem
with re-normalized parameters. We also note that this observation in fact also justifies the
use of single temperature model as effective models for two phase materials
in this asymptotic regime.

The paper is organized as follows: in section 2 we give some basic heuristic arguments and present numerical examples
which clarify the  main results. Sections 3 and 4 are dedicated to the proof of Theorems
\ref{t:2} and \ref{t:1} respectively.

\section{Heuristic arguments and numerical example}\label{sect:Heur}

In this section we would like to  give some formal arguments  and provide  results of numerical simulations of problems \eqref{eq:PT},
\eqref{eq:P} that clarify and illustrate results of
Theorems \ref{t:1} and \ref{t:2}.

Theorem \ref{t:1}  basically states that the presence or absence of global solutions
for  problem \eqref{eq:PT}
is fully determined by the presence or absence of  solutions for  system \eqref{eq:P}.
Thus, the behavior of solutions for   system \eqref{eq:PT} is essentially similar to the behavior
of solutions for single equation \eqref{eq:G}. As it is well known, the dynamics of
a system of parabolic equations, in  general,  is substantially more complex than the
one of a single equation. However, in our case the system \eqref{eq:G} is, at least formally,
can be written as a gradient flow
\begin{eqnarray}
U_t=-\frac{\delta E}{\delta U}, \quad \alpha V_t=-\frac{\delta E}{\delta V}
\end{eqnarray}
with the "energy" functional defined as
\begin{eqnarray}\label{energy}
{\mathcal E}(U,V)=\frac12 \int_{\Omega}\left[ |\nabla U|^2+d|\nabla V|^2+\nu(U-V)^2+2\lambda\Phi(U) \right]dx,
\quad  \Phi(U)=-\int_0^U g(s)ds.
\end{eqnarray}
Moreover, the system \eqref{eq:PT} is  {\em quasi-monotone}  (quasi-monotone nondecreasing in terminology of \cite{Pao}) and thus its classical solutions obey the component-wise parabolic comparison principle \cite{Pao}*{Theorem 3.1, p. 393} or \cite{smith}*{Theorem 3.4, p.130}. Thus the time evolution for solutions of \eqref{eq:PT} is very much restricted
and indeed expected to be similar to the one of a single equation.
This situation is somewhat similar to the one considered in \cite{PG10} where
self explosion in confined porous media was considered.

Now let us turn to Theorem \ref{t:2}.  First we note that transition from existence to non-existence
of solution in this problem is very similar to the one observed in the classical Gelfand problem.
This is again occurs due to the presence of a component-wise comparison principle
and the fact that system \eqref{eq:P} is the Euler-Lagrange equation of the functional
\begin{eqnarray}
{\mathcal E}(u,v)=\frac12 \int_{\Omega}\left[ |\nabla u|^2+d|\nabla v|^2+\nu(u-v)^2+2\lambda\Phi(u) \right]dx,
\quad  \Phi(u)=-\int_0^u g(s)ds.
\end{eqnarray}
In order to understand monotonicity with respect to  parameter $\nu$ it is convenient to rewrite
 system \eqref{eq:P} as a non-local equation.
Combining the first and the second equations of the system \eqref{eq:P} we have
\begin{eqnarray}\label{e-gamma}
[-\gamma\nu^{-1} \Delta+1](u-v)=\lambda\gamma\nu^{-1}g(u)
\end{eqnarray}
where $\gamma=\frac{d}{1+d}$.
Thus,
\begin{eqnarray}\label{int-rep-v}
u-v=\lambda\gamma\nu^{-1}[-\gamma\nu^{-1}\Delta+1]^{-1}g(u).
\end{eqnarray}
Substituting this expression in the first equation of the system \eqref{eq:P}
yields
\begin{eqnarray}\label{int-rep}
-\Delta u=\frac{\lambda}{1+d}\left(1+d\left\{1-[-\gamma\nu^{-1}\Delta +1]^{-1}]\right\}\right) g(u).
\end{eqnarray}
As it is readily seen the operator in a curly brackets is an increasing function of $\nu^{-1}$.
This implies that the effective right hand side of this equation is a decreasing function of  $\nu$ and thus one may expect that
$u$ decreases as $\nu$ increases.
Moreover, in the limiting case $\nu\to\infty$ the right hand side becomes essentially local.
Let us also note that for sufficiently large $\nu$ the components $u$ and $v$ of the system
\eqref{eq:P} away from the boundary are related (at least formally)  by a  simple formula
$v=u -\nu^{-1}\lambda d g(u)/(1+d)+o(\nu^{-1})$  which follows directly from \eqref{int-rep-v}.


\begin{figure}[h]
  \centering
  \includegraphics[width=3.25in]{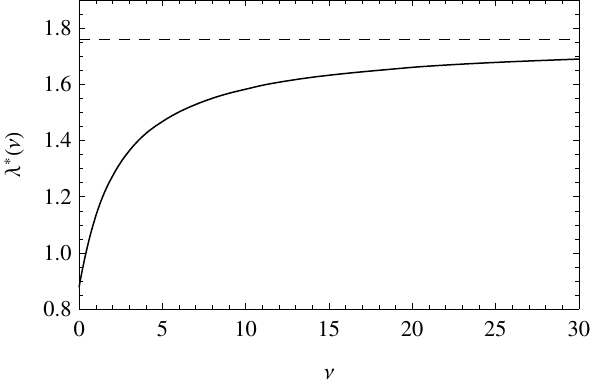}
  \caption{The critical value  $\lambda^*$ as a function of $\nu$ for solution of \eqref{eq:P}  with $\Omega=(-1,1)$, $g(u)=e^u$ and $d=1$. Dashed line represents $\lambda^*(\infty)\approx 1.76$.}
  \label{f:lnu}
\end{figure}
\begin{figure}[h]
  \centering
  \includegraphics[width=6.0in]{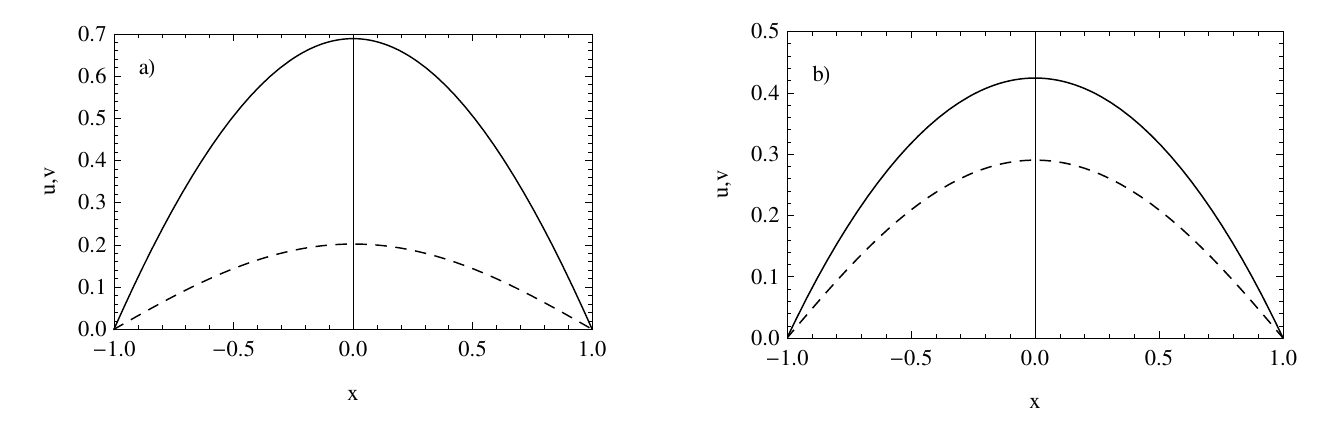}
  \caption{Minimal solution of \eqref{eq:P}  with $\Omega=(-1,1)$ and $g(u)=e^u$, $d=1$ and $\lambda=1$ with $\nu=1$ (fig a) and $\nu=5$ (fig b). $u$ is solid line, $v$ is dashed line. }
  \label{f:uv}
\end{figure}
\begin{figure}[h]
\noindent \begin{minipage}{.4\textwidth}
\includegraphics[width=3.in]{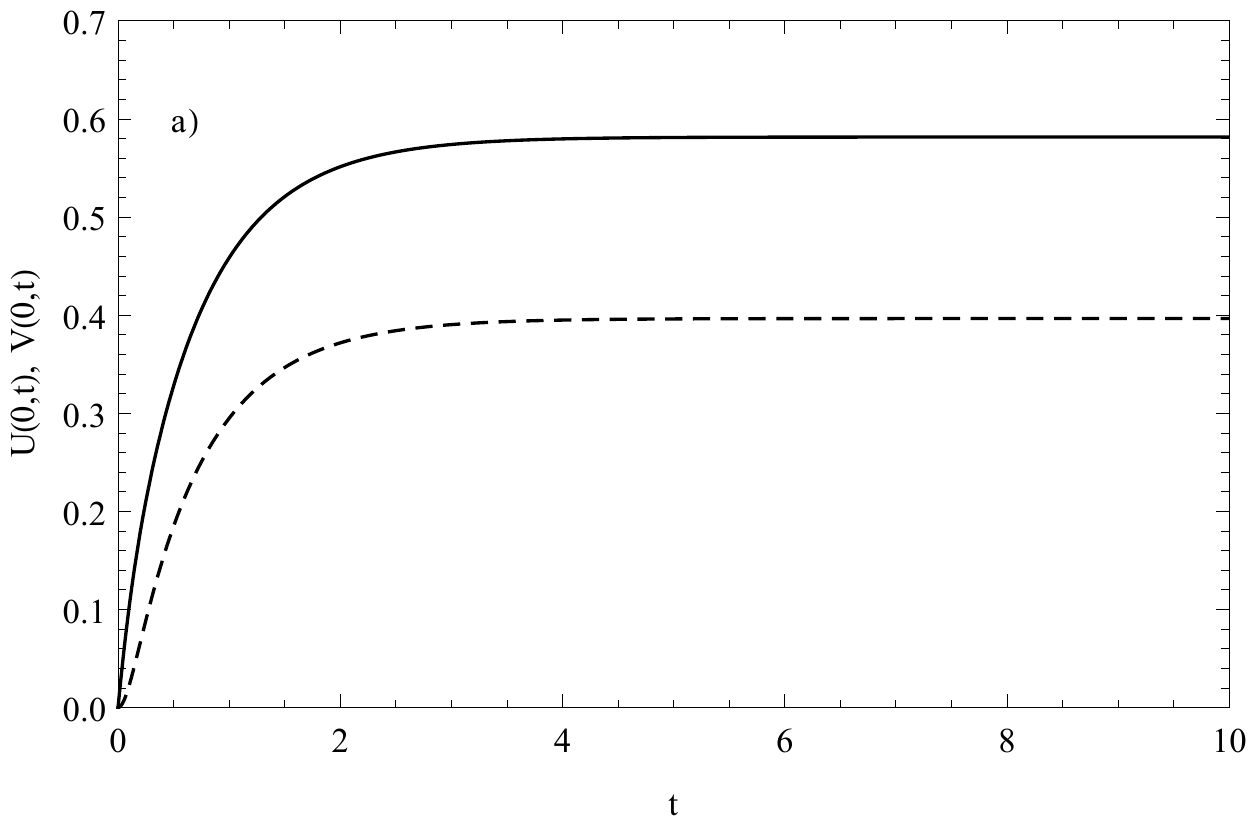}
\end{minipage}
\hspace{1.5cm}
\noindent\begin{minipage}{.4\textwidth}
\includegraphics[width=3.in]{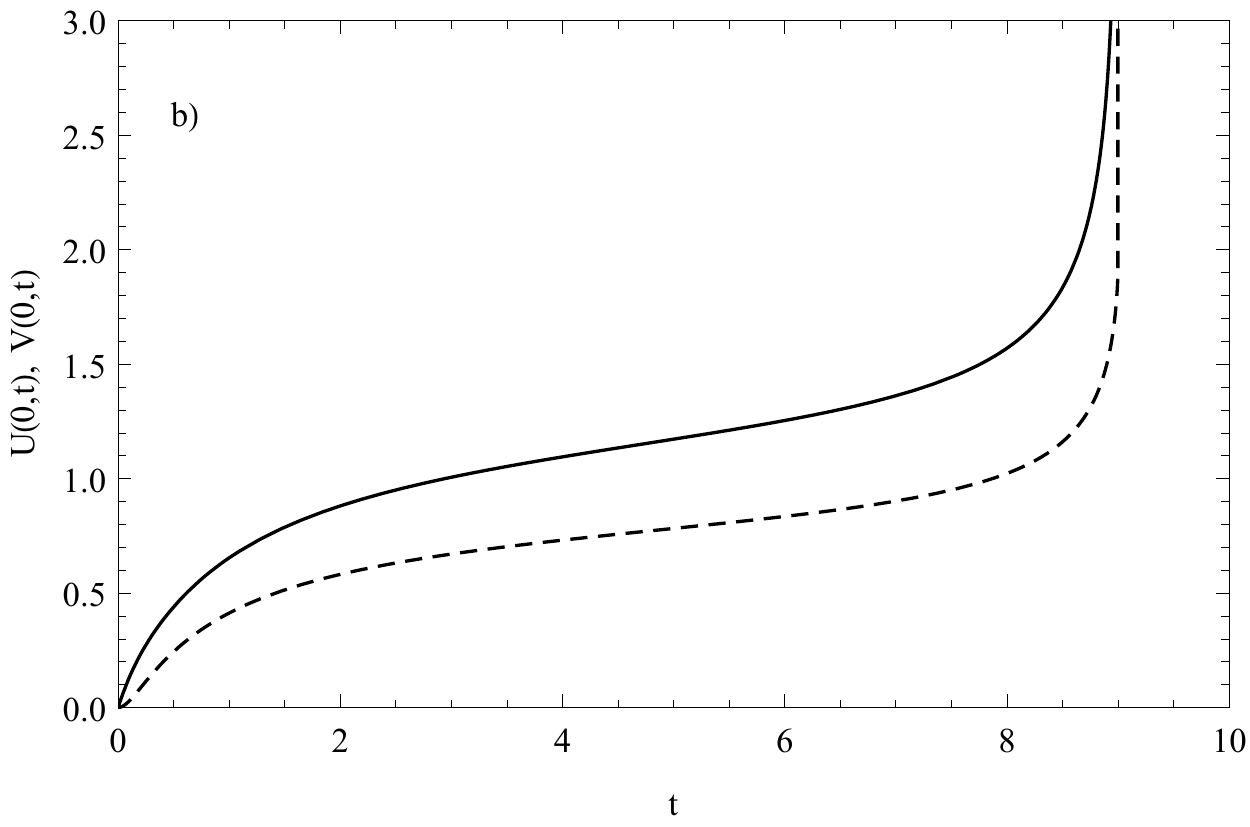}
\end{minipage}
\caption{Solution of  \eqref{eq:PT}  with $\Omega=(-1,1)$ and $g(u)=e^u$, $d=1$, $\alpha=1$,
$\nu=5$ and $\lambda=1.2$ (a) and $\lambda=1.5$ (b) in the middle of an interval $\Omega$ where solution has its maximum  value as long as it exists. $U(0,t)$ is solid line, $V(0,t)$ is dashed line. }
  \label{f:uvt}
\end{figure}

In order to illustrate statements of Theorems \ref{t:1}, \ref{t:2}  we performed numerical studies of a simple one dimensional
versions of  problems \eqref{eq:PT}, \eqref{eq:P} with $\Omega=(-1,1)$, $d=\alpha=1$ and $g(u)=e^u$.
Physically these two problems describe stationary temperature distributions and evolution of temperature fields in
plane-parallel vessel under assumption of Arrhenius chemical kinetics.

Let us first consider the stationary problem \eqref{eq:P}. Solution of this problem was obtained numerically using the
conventional shooting method. The numerical study shows that, in a full agreement with
the  statement of Theorem \ref{t:2},  stationary temperature distribution exists only for
values of scaling parameter $\lambda(\nu)$ which does not exceed some critical value $\lambda^*(\nu)$.
Moreover, this critical value $\lambda^*(\nu) $ is an increasing function of the heat exchange parameter $\nu$ and has  following asymptotic properties:
$\lambda^*(0)=\Lambda^*$,  where $\Lambda^*$ is a critical value of the classical Gelfand problem
\eqref{eq:GS} (in the considered case $\Lambda^*\approx 0.88,$ see e.g \cite{FK}) and $\lambda^*(\infty)=(1+d) \Lambda^*\approx 1.76$.  The dependency of  the critical value $\lambda^*$ as a function of $\nu$ is shown on Figure \ref{f:lnu}.
In addition   for a fixed value of $\lambda<\lambda^* $  one can see that $u$ component describing the temperature of the solid phase is
decreasing monotonically as $\nu$ increases, while  the temperature of the gas phase ($v$ component) is bounded from above by $u$
and approaches to the latter from below as $\nu$ increases see Figure \ref{f:uv}.

Finally, let us  illustrate dynamical features of the combustion process given  by the theorem \ref{t:1}. For this reason    let us consider a  one dimensional problem \eqref{eq:PT} with all the parameters as above and
for a fixed value of $\nu=5$ and two values of $\lambda=1.2$ and $\lambda=1.5$, one of
which is below critical $\lambda^*(5)\approx1.468$ and the other is above critical.
Figure \ref{f:uvt} shows the time evolution of temperatures of gas and solid phases in the middle of the vessel $x=0$ where both temperatures
of the solid and the gas have its maximal values as long as the solution exists.  As predicted by Theorem \ref{t:1} in case of sub-critical $\lambda$  (Figure \ref{f:uvt}a ) the solution, after some short transition period, approaches to
its steady state, whereas for supercritical $\lambda$ (Figure \ref{f:uvt} b) the solution  rapidly accelerates  and
becomes infinite (blows up) in finite time.


\section{Stationary problem: Proof of Theorem \ref{t:2}}\label{Elliptic}

In this section we study solutions of the stationary system \eqref{eq:P} and discuss their qualitative properties.

First, let us note that if $\lambda=0$ then system \eqref{eq:P} becomes  linear
and has a cooperative structure for all $d,\nu>0$.
Further, since $g(u)$ is a monotone non-decreasing function, \eqref{eq:P} is a quasi-monotone non-decreasing nonlinear system
in the sense of \cite{Pao}*{Theorem 4.1, p.406} for every $d,\nu,\lambda>0$ and thus can be studied using comparison type of arguments.

We start with definitions of weak solution of problem \eqref{eq:P} as well as weak sub and super-solutions for this system.

Similarly to \cite{Brezis-96}, we say $(u,v)$ is a {\em weak} solution of system \eqref{eq:P} if $u,v\in L^1(\Omega)$,
$g(u)\delta(x)\in L^1(\Omega)$, where $\delta(x):=\mathrm{dist}(x,\partial\Omega)$,
and
\begin{eqnarray}\label{eq:weak_sol}
&&-\int_\Omega u\Delta\phi+\nu\int_\Omega u\phi-\nu\int_\Omega v\phi=\lambda\int_\Omega g(u)\phi, \nonumber\\
&&-d\int_\Omega v\Delta\psi+\nu\int_\Omega v\psi-\nu\int_\Omega u\psi=0, \qquad  \forall \phi,\psi\in C^2_0(\overline\Omega).
\end{eqnarray}
Note that the assumption $\varphi\in C^2_0(\overline\Omega)$ implies $|\phi|\le C\delta$ for some constant $C>0$,
so the integral on the right hand side of the 1st equation is well--defined.
Note also that zero boundary data are encoded in this definition since we allow test functions $\phi$, $\psi$ which have a nontrivial normal derivative on the boundary.

We say $(u,v)$ is a {\em classical} solution of \eqref{eq:P} if $(u,v)$ is a weak solution of \eqref{eq:P} and in addition,
$u,v\in C^2(\Omega)\cap C_0(\overline\Omega)$.
As usual, $(u,v)$ is a {\em sub} or {\em super}-solution of system \eqref{eq:P} if $=$ above is replaced by $\le$ or $\ge$, respectively,
and in addition only non-negative test functions $\phi$ and $\psi$ are considered.

Given two pairs of functions $(u_1,v_1)$ and $(u_2,v_2)$ defined on $\Omega$, we write $(u_1,v_1)\le(u_2,v_2)$ provided that $u_1(x)\le u_2(x)$ and $v_1(x)\le v_2(x)$ for all $x\in\Omega$.
We say that $(u,v)$ is a {\em minimal} (super-)\,solution of \eqref{eq:P},
if $(u,v)$ is a (super-)\,solution of \eqref{eq:P} and $(u,v)\le(\tilde u,\tilde v)$ for every other super-solution $(\tilde u,\tilde v)$ of \eqref{eq:P}.

Using these definitions we now can proceed to a proof of Theorem \ref{t:2}.
The proof of parts $(i-iii)$ are relatively standard and can be viewed as an extension  of similar results of \cites{Brezis-96,Brezis-97,Cazenave} obtained for single equation, to the system of equations of considered class.
For completeness  we sketch the main steps of the proof of $(i-iii)$.
The proofs of parts $(iv)$ and $(v)$ of Theorem \ref{t:2} are new and will be given in details.
We start with the proof of part $(i)$ of Theorem \ref{t:2}.

\medskip

\begin{proof}
[Proof of part (i) of Theorem \ref{t:2}]
Until the proof of part $(iv)$ of Theorem \ref{t:2} we assume that $\nu>0$ is fixed and when there is no ambiguity,
drop the subscript $\nu$ in the notations.

Let $\mu_1=\mu_1(-\Delta,\Omega)>0$ and $\phi_1> 0$ be the principal eigenvalue
and the corresponding principal eigenfunction of $-\Delta$ in $H_0^1(\Omega)$ with $||\phi||_1=1$.
Recall that since $\Omega$ is smooth,
\begin{equation}\label{phi-delta}
c\delta\le\phi_1\le C\delta,
\end{equation}
for some $C>c>0$, cf. \cite{Cazenave}*{Theorems 3.1.4 and 4.3.1}.

Given $\lambda>0$, let $(u_\lambda,v_\lambda)$ be a weak solution of \eqref{eq:P}.
Testing \eqref{eq:P} against $\phi_1$, in the 2nd equation we obtain
\begin{eqnarray}
\mu_1 d\int_\Omega v_\lambda\phi_1+\nu\int_\Omega v_\lambda\phi_1-\nu\int_\Omega u_\lambda\phi_1=0,
\end{eqnarray}
or
\begin{equation}\label{v-test}
\int_\Omega v_\lambda\phi_1=\frac{\nu}{\nu+\mu_1 d}\int_\Omega u_\lambda\phi_1.
\end{equation}
Substituting into the 1st equation we then derive
\begin{equation}\label{u-test}
\mu_1(1+\kappa)\int_\Omega u_\lambda\phi_1=
\mu_1\int_\Omega u_\lambda\phi_1+\nu\int_\Omega u_\lambda\phi_1-\nu\int_\Omega v_\lambda\phi_1=\lambda\int_\Omega g(u_\lambda)\phi_1,
\end{equation}
where $\kappa=\frac{\nu\,d}{\nu+\mu_1 d}$.
Since $g(0)>0$ and $g$ is convex, by assumption \eqref{eq:g} there is a constant $\eta>0$ such that
\begin{eqnarray}
g(u)\ge \eta u\quad\text{for all}\quad u\ge 0.
\end{eqnarray}
Then from \eqref{u-test} we conclude that
\begin{equation}\label{u-test-2}
\mu_1(1+\kappa)\int_\Omega u_\lambda\phi_1=\lambda\int_\Omega g(u_\lambda)\phi_1
\ge
\lambda\eta\int_\Omega u_\lambda\phi_1.
\end{equation}
This implies that for
\begin{equation}\label{no-weak}
\lambda>\frac{\mu_1}{\eta}\big(1+\kappa)
\end{equation}
system \eqref{eq:P} has no weak solutions.
\end{proof}

\medskip

To prove part $(ii)$ of Theorem \ref{t:2} we need two following lemmas.

\medskip

\begin{lemma}\label{l:new}
Let $(\phi,\psi)$ be a solution of the following problem
\begin{equation}\label{lin-coop}
\left\{
\begin{array}{rcl}
-\Delta \phi+\nu \phi -\nu \psi &=& f\quad\text{in $\Omega$},\\
-d\Delta \psi + \nu \psi - \nu \phi &=& 0\quad\text{in $\Omega$},\\
\phi = \psi &=& 0\quad\text{on $\partial\Omega$}.\\
\end{array}
\right.
\end{equation}
Then,  for every $f\in C(\bar\Omega)$ system \eqref{lin-coop} has unique classical solution $(\phi,\psi)$.
Moreover, $(\phi,\psi)\ge(0,0)$ provided that $f\ge 0$.
In addition, classical solutions of \eqref{lin-coop} satisfy a strong maximum principle, in the sense that
$f\ge 0$ and $f\neq 0$ implies that for some $c,C>0$ it holds
\begin{equation}\label{strong-max}
c\delta(x)\le\psi(x)<\phi(x)\le C\delta(x)\quad\text{for all $x\in\Omega$.}
\end{equation}
\end{lemma}
\begin{proof}
The  existence and uniqueness as well as the regularity and positivity properties for systems of type \eqref{lin-coop}
follow from well known results of \cites{Mitidieri-86,Sweers-92,Fleckinger-95}.

To prove \eqref{strong-max} we observe that combining   the first and the second equations of \eqref{lin-coop} in
a way identical to the one discussed in the section \ref{sect:Heur} (see Eq. \eqref{e-gamma}) we obtain
\begin{eqnarray}\label{eq:100}
[-\gamma\nu^{-1} \Delta+1](\phi-\psi)=\gamma f,
\end{eqnarray}
where $\gamma=\frac{d}{1+d}>0$. By the strong maximum principle  (cf.\ \cite{Cazenave}*{Theorem 3.1.4})
applied to \eqref{eq:100}
we have  $\phi(x)-\psi(x)>c_1\delta(x)$ for all $x\in\Omega$, for some constant $c_1$. Substituting this into the 2nd equation of the system \eqref{lin-coop} and using the strong maximum principle again we obtain the lower bound in \eqref{strong-max}, while from the 1st equation of \eqref{lin-coop} we derive the upper bound of \eqref{strong-max}
via \cite{Cazenave}*{Theorem 4.3.1}.
\end{proof}

\medskip

\begin{lemma}\label{l-super}
Assume that for some $\lambda_*>0$ system \eqref{eq:P} has a classical supersolution.
Then \eqref{eq:P} has a minimal classical solution $(u_\lambda,v_\lambda)$ for every $0<\lambda\le\lambda_*$,
\end{lemma}

\begin{proof}
Let us first observe  that if $(\bar u,\bar v)$ is a classical super-solution of \eqref{eq:P} for some $\lambda_*>0$,
then $(\bar u,\bar v)$ is also a classical super-solution of \eqref{eq:P} for every $0<\lambda\le\lambda_*$.

Next, given $0<\lambda\le\lambda_*$, set $(\phi_0,\psi_0)=(0,0)$. For $k\in\N$, recursively define $(\phi_k,\psi_k)$ as the
unique positive solution of the linear system
\begin{eqnarray}\label{eq:phipskik}
\left\{
\begin{array}{rcl}
-\Delta \phi_k+\nu \phi_k -\nu \psi_k &=& \lambda g(\phi_{k-1})\quad\text{in $\Omega$},\\
-d\Delta \psi_k + \nu \psi_k - \nu \phi_k &=& 0\quad\text{in $\Omega$},\\
\phi_k = \psi_k &=& 0\quad\text{on $\partial\Omega$}.\\
\end{array}
\right.
\end{eqnarray}
By Lemma \ref{l:new} it is clear that
\begin{eqnarray}
0\le \phi_{1}(x)\le \bar u(x),\qquad 0\le \psi_{1}(x)\le \bar v(x)\qquad(x\in\Omega).
\end{eqnarray}
Assume that for some $k\in\N$ it holds
\begin{eqnarray}
0\le \phi_{k-1}(x)\le \phi_{k}(x)\le \bar u(x),\qquad 0\le \psi_{k-1}(x)< \psi_{k}(x)\le \bar v(x)\qquad(x\in\Omega).
\end{eqnarray}
Then, taking into account monotonicity of $g$, we obtain
\begin{eqnarray}\label{eq:phipskik-1}
\left\{
\begin{array}{rcl}
-\Delta (\phi_{k+1}-\phi_k)+\nu (\phi_{k+1}-\phi_k) -\nu(\psi_{k+1}-\psi_k) &=&
\lambda\big(g(\phi_{k})-g(\phi_{k-1})\big)\ge 0\quad\text{in $\Omega$},\\
-d\Delta (\psi_{k+1}-\psi_k) + \nu(\psi_{k+1}-\psi_k) - \nu (\phi_{k+1}-\phi_k) &=& 0\quad\text{in $\Omega$},\\
\phi_{k+1}-\phi_k=\psi_{k+1}-\psi_k &=& 0\quad\text{on $\partial\Omega$}.\\
\end{array}
\right.
\end{eqnarray}
By Lemma \ref{l:new} and the principle of mathematical induction we conclude that
the sequence $(\phi_k,\psi_k)$ is monotone non-decreasing. Similarly,
we deduce $(\phi_k,\psi_k)$ is uniformly bounded by $(\bar u,\bar v)$, so that for all $k\in\N$
it holds
\begin{eqnarray}
0\le \phi_{k}(x)\le \phi_{k+1}(x)\le \bar u(x),\qquad 0\le \psi_{k}(x)\le \psi_{k+1}(x)\le \bar v(x)\qquad(x\in\Omega).
\end{eqnarray}
Therefore, the sequence $(\phi_k,\psi_k)$ converge pointwisely in $\Omega$, and we denote
\begin{eqnarray}
u_\lambda(x):=\lim_{k\to\infty}\phi_k(x)\le \bar u(x),\qquad v_\lambda(x):=\lim_{k\to\infty}\psi_k(x)\le \bar v(x)\qquad(x\in\Omega).
\end{eqnarray}
 By the standard elliptic regularity (cf.\ \cite{Cazenave}*{Proof of Theorem 3.3.3, Step 3}),
 $(u_\lambda,v_\lambda)$ is a classical solution
of the nonlinear system \eqref{eq:P}. Moreover, since the construction of $(u_\lambda,v_\lambda)$
does not depend on the specific choice of a super-solution $(\bar u,\bar v)$,
we conclude that $(u_\lambda,v_\lambda)$ is a minimal solution of \eqref{eq:P}.
\end{proof}
\medskip

Now we turn to the proof of part $(ii)$ of Theorem \ref{t:2}.

\medskip

\begin{proof}[Proof of part (ii) of Theorem \ref{t:2}]
Let $\Lambda^*$ be the critical value of the classical Gelfand problem \eqref{eq:GS}.
For $0<\Lambda<\Lambda^*$, let $u_0:=w_\Lambda$ be the minimal classical solution of \eqref{eq:GS}.
Since $g$ is positive and monotone nondecreasing, it is clear that $(u_0,u_0)$
is a classical super-solution of the nonlinear system \eqref{eq:P} for every $0<\lambda\le\Lambda$.
It follows from Lemma \ref{l-super} and upper bound \eqref{no-weak} in the proof of part $(i)$ of Theorem \ref{t:2} that the set of $\lambda>0$ where \eqref{eq:P} has a minimal classical solution is a bounded, nonempty interval. Thus we defined
\begin{eqnarray}\label{eq:star}
\lambda^*:=\sup\{\lambda>0 : \text{\eqref{eq:P} has a minimal classical solution}\}.
\end{eqnarray}
This completes the proof of parts $(i)$ and $(ii)$ of Theorem \ref{t:2}.
\end{proof}

\medskip

As a next step we continue to the proof of part iii) of the main theorem.

\smallskip

\begin{proof}
[Proof of the claim iii) of Theorem \ref{t:2}]
Given $0<\lambda<\lambda^*$ and $0<\eps<\lambda$, we observe that
$(u_\lambda,v_\lambda)$ is a super-solution of (\ref{eq:P}) with $\lambda$ replaced by ${\lambda-\eps}$.
Therefore, $(u_{\lambda-\eps},v_{\lambda-\eps})\le (u_\lambda,v_\lambda)$.
Let $\phi:=u_\lambda-u_{\lambda-\eps}$, $\psi:=u_\lambda-u_{\lambda-\eps}$.
Then, taking into account the monotonicity of $g$ we see that
\begin{eqnarray}
\left\{\begin{array}{rcll}
- \Delta \phi-\nu(\psi-\phi)&=&\lambda g(u_\lambda)-(\lambda-\eps)g(u_{\lambda-\eps})> 0
& \mbox{in} \quad \Omega, \\
- d\Delta \psi-\nu (\phi-\psi)&=&0 &  \mbox{in} \quad \Omega, \\
\phi=\psi&=&0& \mbox{on} \quad \partial \Omega,
\end{array}\right.
\end{eqnarray}
By the strong maximum principle of \cite{Cazenave}*{Theorem 3.1.4} we conclude that for some $c_\eps>0$ it holds
\begin{eqnarray}
u_\lambda(x)\ge u_{\lambda-\eps}(x)+c_\eps\delta(x),\qquad v_\lambda(x)\ge v_{\lambda-\eps}(x)+c_\eps\delta(x)\qquad (x\in\Omega).
\end{eqnarray}
In particular, $u_\lambda(x)$ and $v_\lambda(x)$ are strictly monotone increasing functions of $\lambda$, for every $x\in\Omega$.

Furthermore, since $g(0)>0$ and $g$ is convex, by assumption \eqref{eq:g} there is a constant $m_*>0$ such that
for all $s\ge 0$,
\begin{eqnarray}
\frac{\lambda^*}{2}g(s)\ge \mu_1(1+\kappa) s-m_*.
\end{eqnarray}
Now, testing \eqref{eq:P} against $\phi_1$ and using \eqref{v-test} and \eqref{u-test} we obtain
\begin{eqnarray}
\lambda\int_\Omega g(u_\lambda)\phi_1=\mu_1(1+\kappa)\int_\Omega u_\lambda\phi_1
\le\frac{\lambda^*}{2}\int_\Omega g(u_\lambda)\phi_1+m_*\int\phi_1.
\end{eqnarray}
We conclude that
\begin{equation}\label{gu-bound}
\lim_{\lambda\to\lambda^*}\int_\Omega g(u_\lambda)\phi_1\le2\frac{m^*}{\lambda^*}<\infty.
\end{equation}
Let $\zeta$ be unique positive solution of the following problem
\begin{eqnarray}\label{torsion1d}
\left\{
\begin{array}{rclll}
-\Delta \zeta&=&1& \mbox{in}& \Omega, \\
\zeta&=&0& \mbox{on} &\partial \Omega.
\end{array}\right.
\end{eqnarray}
Similarly to \eqref{phi-delta}, by \cite{Cazenave}*{Theorems 3.1.4 and 4.3.1} we conclude that
\begin{equation}\label{zeta-delta}
c\delta\le\zeta\le C\delta.
\end{equation}
Testing \eqref{eq:P} with $\phi=\psi=\zeta$, we obtain
\begin{eqnarray}
&&\int_\Omega u_\lambda+\nu\int_\Omega (u_\lambda-v_\lambda)\zeta=\lambda\int_\Omega g(u_\lambda)\zeta,\nonumber \\
&& d\int_\Omega v_\lambda+\nu\int_\Omega (v_\lambda-u_\lambda)\zeta=0,
\end{eqnarray}
Adding these equations together and taking into account \eqref{phi-delta}, \eqref{zeta-delta} and \eqref{gu-bound},
we have
\begin{eqnarray}
\int_\Omega (u_\lambda+d v_\lambda)=\lambda\int_\Omega g(u_\lambda)\zeta\le c_1\lambda\int_\Omega g(u_\lambda)\phi_1<\infty.
\end{eqnarray}
In a view of positivity of $u_\lambda$ and $v_\lambda$ we conclude that $u_{\lambda}$ and $v_{\lambda}$ are bounded in $L^1(\Omega)$.

Since $u_\lambda(x)$ and $v_\lambda(x)$ are both increasing in $\lambda$,
we conclude that $(u_\lambda,v_\lambda)$ converge to $(u^*,v^*)$ in $L^1(\Omega)$,
and $g(u_\lambda)$ converge to $g(u_*)$ in $L^1(\Omega,\delta(x)dx)$.
Similarly to \cite{Brezis-96}*{Lemma 5}, it follows that $(u^*,v^*)$ is a weak solution of \eqref{eq:P} with $\lambda=\lambda^*$.
\end{proof}

\smallskip
We now proceed to the proof of the final two parts of Theorem \ref{t:2}.
First, we will study monotonicity properties of the minimal solutions $u_{\lambda,\nu}$ and $v_{\lambda,\nu}$ with respect to $\nu$, then we consider the limiting behavior of the solutions as $\nu\to 0$,
and finally limiting behavior of the solutions as $\nu\to \infty$.
\medskip

The following lemma establishes monotonicity of $u$ component with respect to parameter $\nu$.

\medskip

\begin{lemma}\label{l:monot}
$\lambda_\nu^*=\lambda^*(\nu)$ is a nondecreasing function of $\nu>0$.
Moreover, for $0<\nu<\tilde \nu$ and $\lambda<\lambda_\nu^*$, let $u_{\lambda,\nu}$ and $u_{\lambda,\tilde\nu}$ be the first components of the minimal solutions of problem \eqref{eq:P}.
Then $u_{\nu}\ge u_{\tilde \nu}$ in $\Omega$.
\end{lemma}

\begin{proof}
Recall that the minimal solution $(u_{\lambda,\nu},v_{\lambda,\nu})$ was constructed by iterations given by system \eqref{eq:phipskik}.
Let $(\phi_k,\psi_k)$ and $(\tilde \phi_k,\tilde \psi_k)$ be solutions of \eqref{eq:phipskik} corresponding to $(\lambda,\nu)$ and
$(\lambda,\tilde \nu)$ respectively.
We are going to show that $\nu<\tilde \nu$ implies that $\phi_k>\tilde \phi_k$ in $\Omega$  for each $k$.
As a result, sequences converging point-wise to solutions $u_{\lambda,\nu}$ and $\tilde u_{\lambda,\nu}$ are ordered.
Hence, the limits are ordered.

Subtracting the first and the second equations of the system  \eqref{eq:phipskik} we observe that
\begin{eqnarray}\label{eq:monot1}
-\Delta(\phi_k-\psi_k)+\nu\left(\frac{1+d}{d}\right)(\phi_k-\psi_k)=\lambda g(\phi_{k-1})\quad \mbox{in $\Omega$},
\end{eqnarray}
which implies that
\begin{eqnarray}\label{eq:monot2}
\phi_k>\psi_k \quad \mbox{in $\Omega$},
\end{eqnarray}
for each $k$. By adding the first and the second equations of  the system  \eqref{eq:phipskik} we also have that
\begin{eqnarray}\label{eq:monot3}
-\Delta (\phi_k+d\psi_k)=\lambda g(\phi_{k-1})\quad \mbox{in $\Omega$},
\end{eqnarray}
for each $k$ and $\nu$. Needles to say that identical equations hold for $\tilde \phi_k$ and $\tilde \psi_k$.

Now let us show that $\phi_1>\tilde \phi_1$. Indeed since $\phi_0=\tilde\phi_0=0$ in $\Omega$, we have from \eqref{eq:monot3}
\begin{eqnarray}\label{eq;monot4}
-\Delta(\phi_1+d\psi_1)=-\Delta(\tilde \phi_1+d\tilde \psi_1)=\lambda g(0) \quad \mbox{in $\Omega$},
\end{eqnarray}
and therefore
\begin{eqnarray}\label{eq:monot5}
\phi_1+d\psi_1=\tilde \phi_1+d\tilde \psi_1 \quad \mbox{in $\Omega$}.
\end{eqnarray}
Taking difference of the first equations of the system \eqref{eq:phipskik} for $\nu$ and $\tilde \nu$ after some algebra
we obtain
\begin{eqnarray}\label{eq:monot6}
-\Delta(\phi_1-\tilde \phi_1)+\nu(\phi_1-\tilde \phi_1)-\nu(\psi_1-\tilde \psi_1)=(\tilde \nu-\nu)(\tilde \phi_1-\tilde \psi_1)
\quad \mbox{in $\Omega$}.
\end{eqnarray}
Using  \eqref{eq:monot5} and \eqref{eq:monot2} we obtain from \eqref{eq:monot6}
\begin{eqnarray}\label{eq:monot7}
-\Delta(\phi_1-\tilde \phi_1)+\nu\left(\frac{1+d}{d}\right)(\phi_1-\tilde\phi_1)=(\tilde \nu-\nu) (\tilde \phi_1-\tilde\psi_1)>0 \quad \mbox{in $\Omega$},
\end{eqnarray}
and thus
\begin{eqnarray} \label{eq:monot8}
\phi_1>\tilde \phi_1 \quad \mbox{in $\Omega$}.
\end{eqnarray}

Let us show now that $\phi_k>\tilde \phi_k$ provided $\phi_{k-1}>\tilde \phi_{k-1}$. Indeed, assume
that
\begin{eqnarray}\label{eq:monot9}
\phi_{k-1}>\tilde \phi_{k-1} \qquad \mbox{in} \quad \Omega,
\end{eqnarray}
then by \eqref{eq:monot3} we have
\begin{eqnarray}\label{eq:monot10}
&&-\Delta (\phi_k+d\psi_k)=\lambda g(\phi_{k-1}), \nonumber \\
&& -\Delta(\tilde \phi_k+d\tilde \psi_k)=\lambda g(\tilde \phi_{k-1}) \quad \mbox{in $\Omega$}.
\end{eqnarray}
Using the fact that $g$ is increasing and assumption \eqref{eq:monot9} we have
\begin{eqnarray}\label{eq:monot11}
g(\phi_{k-1})>g(\tilde\phi_{k-1}),
\end{eqnarray}
which together with \eqref{eq:monot10} gives
\begin{eqnarray}\label{eq:monot12}
\phi_k+d\psi_k>\tilde\phi_k+d\tilde\psi_k \quad \mbox{in $\Omega$}.
\end{eqnarray}
Combining first equations of the system  \eqref{eq:phipskik} for $\nu$ and $\tilde \nu$  we have
\begin{multline}\label{eq:monot13}
-\Delta(\phi_k-\tilde\phi_k)+\nu(\phi_k-\tilde\phi_k)-\nu(\psi_k-\tilde\psi_k)=\\
\lambda(g(\phi_{k-1})-g(\tilde\phi_{k-1}))+
(\tilde\nu-\nu)(\tilde \phi_k-\tilde\psi_k)>0 \quad \mbox{in $\Omega$}.
\end{multline}
Note that positivity of the right hand side of the above equation follows from \eqref{eq:monot11} and \eqref{eq:monot2}.

Using \eqref{eq:monot12} from \eqref{eq:monot13} we have
\begin{eqnarray}
-\Delta(\phi_k-\tilde\phi_k)+\nu\left(\frac{1+d}{d}\right)(\phi_k-\tilde\phi_k)>0 \quad \mbox{in $\Omega$},
\end{eqnarray}
that yields
\begin{eqnarray}\label{eq:monot16}
\phi_k>\tilde \phi_k\quad \mbox{in $\Omega$}.
\end{eqnarray}
In a view of \eqref{eq:monot8} an inequality \eqref{eq:monot16} holds for each $k\ge 1$.
Since  $\phi_k\to u$ and $\tilde \phi_k\to \tilde u$ we conclude that
\begin{eqnarray}
u_{\lambda,\nu}\ge \tilde u_{\lambda,\tilde\nu} \quad \mbox{in $\Omega$}.
\end{eqnarray}
By construction, it follows that $\lambda_\nu^*=\lambda^*(\nu)$ is a nondecreasing function of $\nu>0$,
which completes the proof.
\end{proof}

Let us now consider the limiting behavior of the system \eqref{eq:P} as $\nu\to 0$.

\begin{proposition}\label{p:41}
For $\lambda<\Lambda^\ast$ and $\nu\to 0$, the minimal solution $(u_{\lambda,\nu},v_{\lambda,\nu})$
converges uniformly to $(u_0,0)$, where $u_0$ is the minimal solution of \eqref{eq:P0}.
\end{proposition}

\begin{proof}
Let $\lambda<\Lambda_\ast$ and let $u_0$ be the minimal classical solution of Gelfand problem \eqref{eq:P0}.
Then $(u_0,u_0)$ is a supersolution of \eqref{eq:P} and in particular, for all $\nu>0$ holds
\begin{equation}
(u_{\lambda,\nu},v_{\lambda,\nu})\le (u_0,u_0).
\end{equation}
The minimal solution $v_{\lambda,\nu}$ of the second equation of \eqref{eq:P} can be represented as
\begin{eqnarray}
v_{\lambda,\nu}=\nu [-\Delta+\nu]^{-1} u_{\lambda,\nu}.
\end{eqnarray}
Since $u_{\lambda,\nu}$ is uniformly bounded in $L^{\infty}(\Omega)$ and $[-\Delta+\nu]^{-1}$ is
a bounded operator from $L^{\infty}(\Omega)$ into $L^{\infty}(\Omega)$, we have
\begin{eqnarray}
0<v_{\lambda,\nu}\le \nu C \quad \mbox{in $\Omega$},
\end{eqnarray}
for some $C>0$ independent of $\nu$, that is
\begin{eqnarray}
\|v_{\lambda,\nu}\|_{L^\infty}\to 0 \quad \mbox{as} \quad \nu\to 0.
\end{eqnarray}
Next, since $g$ is of class $C^1$, combining first equations of \eqref{eq:P} and \eqref{eq:P0}
and setting $w_\nu:=u_0-u_{\lambda,\nu}$ we obtain
\begin{eqnarray}\label{eq:inv00}
-\Delta w_\nu=\lambda \left( g(u_0)-g(u_{\lambda,\nu})\right)+\nu (u_{\lambda,\nu}-v_{\lambda,\nu})= g^{\prime}(\xi)w_\nu+\nu(u_{\lambda,\nu}-v_{\lambda,\nu})
\end{eqnarray}
where $\xi_\nu\in L^\infty(\Omega)$ satisfy
\begin{eqnarray}\label{order-0-0}
u_{\lambda,\nu}\le\xi_\nu\le u_0 \quad \mbox{in $\Omega$}.
\end{eqnarray}
Since $u_0$ is a minimal solution of corresponding Gelfand problem \eqref{eq:P0}, it is stable \cite{Brezis-96} in the sense
that the operator $-\Delta-\lambda g^{\prime}(u_0)$ is invertible in $L^2(\Omega)$.
In view of \eqref{order-0-0}, the operator $-\Delta-\lambda g^{\prime}(\xi)$ is also invertible in $L^2(\Omega)$.
This allows to rewrite \eqref{eq:inv00} as follows
\begin{eqnarray}
w_\nu=\nu [-\Delta-\lambda g^{\prime}(\xi_\nu)]^{-1}(u_{\lambda,\nu}-v_{\lambda,\nu}).
\end{eqnarray}
Since $[-\Delta-\lambda g^{\prime}(\xi_\nu)]^{-1}$ is bounded from $L^{\infty}(\Omega)$ into $L^{\infty}(\Omega)$,
while $u_{\lambda,\nu}$ and $v_{\lambda,\nu}$ are uniformly bounded in $L^{\infty}(\Omega)$, we conclude that
\begin{eqnarray}
0<w_\nu\le C\nu\quad \mbox{in $\Omega$}.
\end{eqnarray}
Therefore, $\|w_\nu\|_{L^\infty}\to 0$ as $\nu\to 0$ and the assertion follows.
\end{proof}

\bigskip
We now give several lemmas needed to study the behavior of solution for the problem \eqref{eq:P} in the limit of $\nu\to \infty$.
To shorten the notation, we denote
\begin{equation}
\mathbf{K}_\nu:=d\left\{1-\nu[-\gamma\Delta +\nu]^{-1}]\right\},
\end{equation}
where $\gamma=\frac{d}{1+d}$.
Clearly, $\K_\nu$ is a bounded linear operator in $C(\Omega)$ and in $L^p(\Omega)$, for any $1\le p \le \infty$.
Similarly to the derivation of \eqref{int-rep}, we see that if $(u,v)$
is a classical solution of \eqref{eq:P} then $u$ is a classical solution of the nonlocal equation
\begin{equation}\label{eq-nonloc}
-\Delta u=\frac{\lambda}{1+d}\big(1+\K_\nu\big)g(u)=0\quad\text{in }\:\Omega,
\qquad u=0\quad\text{on }\:\partial\Omega.
\end{equation}
and
\begin{eqnarray}\label{int-rep-v-nonloc}
v=u-\lambda\gamma[-\gamma\Delta+\nu]^{-1}g(u).
\end{eqnarray}

We present two standard results about the properties of the operator $\K_\nu$.

\begin{lemma}\label{K-nu-positive}
For all $\nu>0$, $\K_\nu$ is a positive operator in $L^\infty(\Omega)$, i.e. for every $f\in L^\infty(\Omega)$, $f\ge 0$ implies $\K_\nu f\ge 0$. Moreover, $\|\K_\nu\|_{L^2\to L^2}=1$ and $\K_\nu$ strongly converges to zero as $\nu\to\infty$,
i.e. for every $f\in L^2(\Omega)$, $\lim_{\nu\to\infty}\|\K_\nu f\|_{L^2}= 0$.
\end{lemma}

\begin{proof}
Observe that the resolvent operator $[-\gamma\Delta +\nu]^{-1}$ is well-defined in $L^2(\Omega)$ for all $\nu>0$,
and by spectral theorem,
\begin{equation}
\|\nu[-\gamma\Delta +\nu]^{-1}\|_{L^2\to L^2}=\nu(\gamma\mu_1+\nu)^{-1}<1,
\end{equation}
that is $\nu[-\gamma\Delta +\nu]^{-1}$ is a {\em contraction} in $L^2(\Omega)$.
Then by \cite{Ma}*{Proposition 1.3}, the family $[-\gamma\Delta +\nu]^{-1}$
is a {\em strongly continuous contraction resolvent}, that is
\begin{equation}
\|f-\nu[-\gamma\Delta +\nu]^{-1}f\|_{L^2}\to 0\quad\text{for every } f\in L^2(\Omega).
\end{equation}
Moreover, by \cite{Ma}*{Definition 4.1 and Chapter 2.1},
$[-\gamma\Delta +\nu]^{-1}$ is sub-Markovian, that is for all $f\in L^2(\Omega)$,
\begin{equation}
0\le f\le 1\quad\text{implies}\quad 0\le\nu [-\gamma\Delta +\nu]^{-1}f\le 1.
\end{equation}
Since $\K_\nu=\mathbf{I}-\nu[-\gamma\Delta +\nu]^{-1}$, we conclude that $\K_\nu$ is a positive operator
in $L^\infty(\Omega)$ for all $\nu>0$, and that $\|\K_\nu f\|_{L^2}\to 0$ for every $f\in L^2(\Omega)$.
\end{proof}

\begin{lemma}\label{stable-nu}
Let $\lambda<\lambda^\ast_{\nu_0}$ for some $\nu_0>0$. Then
$\mu_1\big(-\Delta-\frac{\lambda}{1+d}g'(u_{\lambda,\nu})\big)>0$ for all $\nu>\nu_0$.
\end{lemma}

\begin{proof}
Let $(u_{\lambda,\nu},v_{\lambda,\nu})$ be be the minimal positive solution of \eqref{eq:P}.
Consider the eigenvalue problem for the linearized system
\begin{equation}\label{lin-min}
\left\{
\begin{array}{rcl}
-\Delta \phi+\big(\nu- g'(u_{\lambda,\nu})\big) \phi -\nu \psi &=& \mu \phi\quad\text{in $\Omega$},\\
-d\Delta \psi + \nu \psi - \nu \phi &=& \mu\psi\quad\text{in $\Omega$},\\
\phi = \psi &=& 0\quad\text{on $\partial\Omega$}.\\
\end{array}
\right.
\end{equation}
It is known that system \eqref{lin-min} admits the principal eigenvalue $\tilde\mu_{\nu,1}$,
the corresponding eigenfunction $(\phi,\psi)$ can be chosen positive
and $\phi,\psi\in C^2(\Omega)\cap C_0(\bar\Omega)$,
see \cite{Sweers-92}*{Theorem 1.1}.

Rearranging system \eqref{lin-min} as in Section \ref{sect:Heur} we see that
$\tilde\mu_{\nu,1}$ and $\phi>0$ satisfy the nonlocal equation
\begin{equation}
-\Delta\phi-\frac{\lambda}{1+d}(1+\K_\nu)g'(u_{\lambda,\nu})\phi=\tilde\mu_{\nu,1}\phi\quad\text{in }\:\Omega,
\qquad \phi=0\quad\text{on }\:\partial\Omega.
\end{equation}
Similarly to \cite{Cazenave}*{Proposition 3.4.4}, assume that $\tilde\mu_{\nu,1}<0$.
Given $\eps>0$, we compute
\begin{multline}
-\Delta(u_{\lambda,\nu}-\eps\phi)-\frac{\lambda}{1+d}(1+\K_\nu) g(u_{\lambda,\nu}-\eps\phi)\\
=\frac{\lambda}{1+d}(1+\K_\nu)g(u_{\lambda,\nu})-\eps\frac{\lambda}{1+d}(1+\K_\nu)g'(u_{\lambda,\nu})\phi-\tilde\mu_{\nu,1}\eps\phi-
 \frac{\lambda}{1+d}(1+\K_\nu)g(u_{\lambda,\nu}-\eps\phi)\\
=-\frac{\lambda}{1+d}(1+\K_\nu)\big(g(u_{\lambda,\nu}-\eps\phi)-g(u_{\lambda,\nu})+\eps g'(u_{\lambda,\nu})\phi\big)-\tilde\mu_{\nu,1}\eps\phi,
\end{multline}
where
\begin{equation}
(1+\K_\nu)\big(g(u_{\lambda,\nu}-\eps\phi)-g(u_{\lambda,\nu})+\eps g'(u_{\lambda,\nu})\phi\big)=o(\eps\phi),
\end{equation}
since $g$ is $C^1$, $u_{\lambda,\nu},\phi\in L^\infty(\Omega)$ and $1+\K_\nu$ is bounded in $L^\infty(\Omega)$.
Since $\tilde\mu_{\nu,1}<0$, we deduce that
\begin{equation}
-\Delta(u_{\lambda,\nu}-\eps\phi)-\frac{\lambda}{1+d}(1+\K_\nu) g(u_{\lambda,\nu}-\eps\phi)\ge 0
\end{equation}
for all sufficiently small $\eps>0$.

Using the definition of $\K_\nu$, we then conclude that
\begin{multline}
-\Delta(u_{\lambda,\nu}-\eps\phi)+\nu(u_{\lambda,\nu}-\eps\phi-v_{\lambda,\nu})-\lambda g(u_{\lambda,\nu}-\eps\phi)\\=-\Delta(u_{\lambda,\nu}-\eps\phi)-\frac{\lambda}{1+d}(1+\K_\nu) g(u_{\lambda,\nu}-\eps\phi)\ge 0\quad\text{in }\:\Omega,
\end{multline}
and further, we note that
\begin{equation}
-\Delta v_{\lambda,\nu}+\nu(v_{\lambda,\nu}-u_{\lambda,\nu}+\eps\phi)=\eps\phi\ge 0,
\end{equation}
for all sufficiently small $\eps>0$. This means that $(u_{\lambda,\nu}-\eps\phi,v_{\lambda,\nu})$ is a supersolution
of system \eqref{eq:P}. By Lemma \ref{l-super} we conclude that system \eqref{eq:P} admits
a solution $(\bar u_{\lambda,\nu},\bar v_{\lambda,\nu})$ with $0<\bar u_{\lambda,\nu}<u_{\lambda,\nu}$. But this contradicts to the minimality
of $(u_{\lambda,\nu},v_{\lambda,\nu})$. Hence $\tilde\mu_{\nu,1}\ge 0$.

Now observe that
\begin{equation}
\int_\Omega \K_\nu g'(u_{\lambda,\nu})\varphi^2\ge g'(0)\sigma_1(\nu)\int_\Omega\varphi^2,
\end{equation}
where $\sigma_1(\nu)=d\big(1-(\gamma\nu^{-1}\mu_1(-\Delta)+1)^{-1}\big)>0$ is the smallest eigenvalue of $\K_\nu$
in $L^2(\Omega)$. This implies that $\mu_1(-\Delta-g'(u_\nu))>\tilde\mu_{\nu,1}\ge 0$.
\end{proof}
\bigskip

Using  results  presented above we can now describe the limiting behavior of solution for the problem \eqref{eq:P}.
\bigskip

\begin{proposition}\label{p:42}
For $\lambda<\Lambda^\ast(1+d)$ and $\nu\to \infty$, the minimal solution $(u_{\lambda,\nu},v_{\lambda,\nu})$
converges uniformly to $(u_{\infty}, u_{\infty})$, where $u_\infty$ is the minimal solution of \eqref{eq:Pinf}.
\end{proposition}

\begin{proof}
Using representation \eqref{eq-nonloc}, we see that
\begin{eqnarray}\label{eq-nonlocal}
-\Delta u_{\lambda,\nu}=\frac{\lambda}{d+1} g(u_{\lambda,\nu})+\frac{\lambda}{d+1}\K_\nu g(u_{\lambda,\nu})\ge \frac{\lambda}{d+1} g(u_{\lambda,\nu}).
\end{eqnarray}
In view of positivity of $\K_\nu$ (Lemma \ref{K-nu-positive}), we conclude that
$u_{\lambda,\nu}$ is a supersolution of \eqref{eq:Pinf}. Since $u_\infty$ is the minimal solution of \eqref{eq:Pinf},
we see that
\begin{equation}\label{infty-bound}
u_{\lambda,\nu}\ge u_\infty.
\end{equation}
By Lemma \ref{l:monot}, $u_{\nu}$ is monotone decreasing as $\nu\to\infty$, so for a $\tilde\nu>0$ and all $\nu\in[\tilde\nu,\infty)$ we have
\begin{equation}\label{infty-bound-twoside}
u_{\lambda,\tilde\nu}\ge u_{\lambda,\nu}\ge u_\infty.
\end{equation}
In particular, $\{u_{\lambda,\nu}\}_{\nu\ge\tilde\nu}$ and $\{g(u_{\lambda,\nu})\}_{\nu\ge\tilde\nu}$
are bounded in $L^2(\Omega)$, for every $p\in[1,\infty]$.
Since $\K_\nu$ is bounded and $(-\Delta)^{-1}$ is compact in $L^p(\Omega)$ for every $p\in(1,\infty)$, and
\eqref{eq-nonlocal} can be rewritten as
\begin{eqnarray}\label{eq-nonlocal-inverse}
u_{\lambda,\nu}=(-\Delta)^{-1}\Big(\tfrac{\lambda}{d+1} g(u_{\lambda,\nu})+\tfrac{\lambda}{d+1}\K_\nu g(u_{\lambda,\nu})\Big),
\end{eqnarray}
we conclude that for a sequence $\nu_n\to\infty$, $u_{\lambda,{\nu_n}}$ converges to a limit $\tilde u_\infty$.
Moreover, since $u_{\lambda,{\nu}}(x)$ is monotone decreasing in $x$ (Lemma \ref{l:monot}), we conclude that
$\lim_{\nu\to\infty}u_{\lambda,\nu}=\tilde u_\infty$ and $\lim_{\nu\to\infty}g(u_{\lambda,\nu})=g(\tilde u_\infty)$
Using strong convergence of $\K_\nu$ to zero (Lemma \ref{K-nu-positive}), we conclude that
\begin{equation}\label{e:K-nu-converges}
\|\K_\nu g(u_{\lambda,\nu})\|_{L^2}\le\|\K_\nu g(\tilde u_\infty)\|_{L^2}+\|\K_\nu\|_{L^2\to L^2}\|g(u_{\lambda,\nu})-g(\tilde u_\infty)\|_{L^2}\to 0,
\end{equation}
as $\nu\to\infty$.

Combining \eqref{eq-nonloc} and \eqref{eq:Pinf} and setting $w_\nu:=u_{\lambda,\nu}-u_\infty$ we obtain
\begin{multline}\label{eq:inv0}
-\Delta w_\nu=\frac{\lambda}{d+1}\big( g(u_{\lambda,\nu})-g(u_{\infty})\big)+\frac{\lambda}{d+1}\K_\nu g(u_{\lambda,\nu})\\
=\frac{\lambda}{d+1}g'(\xi_\nu)w_\nu+\frac{\lambda}{d+1}\K_\nu g(u_{\lambda,\nu})
\end{multline}
where $\xi_\nu\in L^\infty(\Omega)$ and
\begin{eqnarray}\label{order-0}
u_{\infty}\le\xi_\nu\le u_{\lambda,\nu}.
\end{eqnarray}
By Lemma \ref{stable-nu}, the operator $-\Delta-\frac{\lambda}{1+d}g'(u_{\lambda,\nu})$ is invertible in $L^2(\Omega)$.
Then, in view of \eqref{order-0}, the operator $-\Delta-\frac{\lambda}{1+d} g^{\prime}(\xi_\nu)$
is also invertible in $L^2(\Omega)$. This allows to rewrite \eqref{eq:inv0} as follows
 \begin{eqnarray}
 w_\nu=\tfrac{\lambda}{d+1}\big[-\Delta-\tfrac{\lambda}{1+d} g^{\prime}(\xi_\nu)\big]^{-1}\K_\nu g(u_{\lambda,\nu})
 \end{eqnarray}
 In view of \eqref{e:K-nu-converges} and since the operator $[-\Delta-\tfrac{\lambda}{1+d} g^{\prime}(\xi_\nu)\big]^{-1}$ is bounded in $L^2(\Omega)$, we conclude that $\|w_\nu\|_{L^2}\to 0$ as $\nu\to\infty$ and thus $\|u_{\lambda,\nu}-u_\infty\|_{L^2}\to 0$
 as $\nu\to \infty$. In particular, this implies that $\tilde u_\infty=u_\infty$.

 Further, using the standard bootstrap argument we improve the convergence to conclude
 that $\|w_\nu\|_{L^\infty}\to 0$ as $\nu\to \infty$, and thus $\|u_{\lambda,\nu}-u_\infty\|_{L^\infty}\to0$
 as $\nu\to \infty$.

 Finally, by \eqref{int-rep-v-nonloc},
\begin{eqnarray}
v_{\lambda,\nu}=u_{\lambda,\nu}-\lambda\gamma[-\gamma\Delta+\nu]^{-1}g(u_{\lambda,\nu}).
\end{eqnarray}
Since $[-\gamma\Delta+\nu]^{-1}$ is bounded as operator from $L^{\infty}(\Omega)$ into $L^{\infty}(\Omega)$,
we also conclude that $\|v_{\lambda,\nu}-u_{\lambda,\nu}\|_{L^\infty}\to 0$ as $\nu\to\infty$.
\end{proof}
\bigskip

We are now in a position to complete the proof of Theorem \ref{t:2}.
\smallskip
\begin{proof}[Proof of part (iv) and  (v) of Theorem \ref{t:2}]
The claim follows immediately from Lemma \ref{l:monot} and Propositions \ref{p:41}, \ref{p:42}.
\end{proof}

\section{Parabolic problem: proof of Theorem \ref{t:1}}

In this section we present a proof of Theorem \ref{t:1}. This theorem can be viewed as
an extension of the result obtained in \cite{PG10}  for system which describes thermal
ignition in one-phase confined materials  which in turn is an extension of the
result for classical Gelfand problem given in \cite{Brezis-96}.

In order to proceed we will need a following  lemma.

\begin{lemma}\label{l:1}
Let $U,V$ be a global classical solution of \eqref{eq:PT}. Then, $(U_t,V_t)\ge 0$ in $\Omega$.
\end{lemma}

\begin{proof}
Differentiating system \eqref{eq:PT} with respect to time and setting $\xi=U_t$, $\eta=V_t$ we have:
\begin{eqnarray}\label{eq:l11}
\left\{\begin{array}{ll}
\xi_t- \Delta \xi=\lambda g^{\prime} ( U) \xi+\nu(\eta-\xi),& \\
\alpha \eta_t- d\Delta \eta=\nu (\xi-\eta)
 &  \mbox{in} \quad (0,T)\times \Omega, \\
\xi=\eta=0& \mbox{on} \quad \partial \Omega,\\
\xi(0,\cdot),\; \eta(0,\cdot)\ge 0& \mbox{in}\quad \Omega.
\end{array}\right.
\end{eqnarray}
The linear system \eqref{eq:l11} is quasi-monotone and thus component-wise comparison principle holds \cite{Pao}.
Since $\xi=\eta=0$ is a sub-solution we have $\xi,\eta\ge 0$ for all $t\ge0$ in $\Omega$.
\end{proof}

\medskip

Now we turn to a proof of Theorem \ref{t:1}.

\medskip

\begin{proof}[Proof of Theorem \ref{t:1}]

First we claim that if \eqref{eq:P} has a classical solution, then \eqref{eq:PT} has a global
solution.  This follows directly from the fact  that the system \eqref{eq:PT} is quasi-monotone and thus
a comparison principle holds component-wise \cite{Pao}.

 Now  let us show that existence of global solution
for problem \eqref{eq:PT} imply existence of weak solution for \eqref{eq:P}.

Let us first note that by lemma \ref{l:1} $U,V,U_t,V_t\ge 0$ for all $x\in\Omega$ and $t\ge0$,
so that solutions of the problem \eqref{eq:PT} are non-negative and non-decreasing.

Next, observe that for each $\phi,\psi\in C^2(\bar \Omega)$ with $\phi=\psi=0$ on $\partial
\Omega$ we have
\begin{eqnarray}\label{eq:t11}
&&\frac{d}{dt} \int_{\Omega} U\phi+\int_{\Omega} U(-\Delta \phi)=
\lambda \int_{\Omega}  g(U)\phi +\nu \int_{\Omega} (V-U)\phi, \nonumber \\
&&\alpha \frac{d}{dt} \int_{\Omega} V\psi+d\int_{\Omega} V(-\Delta \psi)=
\nu\int_{\Omega} (U-V) \psi.
\end{eqnarray}
As before, let $\mu_1>0$ be the principal Dirichlet eigenvalue of $-\Delta$ in $\Omega$
and $\phi_1> 0$ be the corresponding eigenfunction, with $||\phi||_1=1$.
Setting $\phi=\psi=\phi_1$ in \eqref{eq:t11} we have
\begin{eqnarray}\label{eq:t12}
&&\frac{d}{dt} \int_{\Omega} U\phi_1+\mu_1\int_{\Omega} U\phi_1=
\lambda \int_{\Omega}  g(U)\phi_1 +\nu \int_{\Omega} (V-U)\phi_1, \nonumber \\
&&\alpha \frac{d}{dt} \int_{\Omega} V\phi_1+\mu_1d\int_{\Omega} V\phi_1=
\nu\int_{\Omega} (U-V) \phi_1.
\end{eqnarray}
We first claim that $\int_{\Omega} U\phi_1$ and $\int_{\Omega} V\phi_1$ are uniformly bounded in
time.

From first equation of \eqref{eq:t12}, non negativity of $V$,  $\phi_1$ and Jensen's inequality
we have
\begin{multline}\label{eq:t13}
\frac{d}{dt} \int_{\Omega} U\phi_1+(\mu_1+\nu)\int_{\Omega} U\phi_1\\=
\lambda \int_{\Omega}  g(U)\phi_1 +\nu \int_{\Omega} V\phi_1\ge \lambda \int_{\Omega}  g(U)\phi_1\ge \lambda  g\left( \int_{\Omega} U\phi_1\right).
\end{multline}
By assumption \eqref{eq:g} we have $g^{\prime}(s)\to\infty$ as $s\to\infty$. Therefore, there is
a constant $M_1>0$ such that
\begin{eqnarray}
\lambda g(s)-(\mu_1+\nu)s\ge \frac{\lambda}{2} g(s), \quad \mbox{for} \quad s\ge M_1.
\end{eqnarray}
Now assume that   $\int_{\Omega} U\phi_1=M_1$ at  $t=t_0$, then for $t\ge t_0$ we have
 \begin{eqnarray}\label{eq:t14}
&&\frac{d}{dt} \int_{\Omega} U\phi_1
\ge \frac{\lambda}{2}  g\left( \int_{\Omega} U\phi_1\right)
\end{eqnarray}
which contradicts \eqref{eq:g} and thus
\begin{eqnarray}\label{eq:t15}
\int_{\Omega} U\phi_1\le M_1.
\end{eqnarray}
In a view or \eqref{eq:t15} we have from the second equation of \eqref{eq:t12}
that
\begin{eqnarray}\label{eq:t16}
&&\alpha \frac{d}{dt} \int_{\Omega} V\phi_1+(\mu_1d+\nu) \int_{\Omega} V\phi_1=
\nu\int_{\Omega} U \phi_1\le \nu M_1
\end{eqnarray}
which immediately imply that
\begin{eqnarray}\label{eq:t17}
\int_{\Omega} V\phi_1\le M_2
\end{eqnarray}
for some constant $M_2>0$.
Finally integrating \eqref{eq:t12} on $(t,t+1)$ and taking into account that $g(U)$ and $U$
are non-decreasing we have
\begin{multline}\label{eq:t18}
\lambda \int_{\Omega} g(U(t))\phi_1\le \int_t^{t+1}  \int_{\Omega} g(U)\phi_1\\
\le \int_{\Omega} U(t+1)\phi_1+(\mu_1+\nu)\int_t^{t+1} \int_{\Omega} U\phi_1\le
(1+\mu_1+\nu)M_1
\end{multline}
and thus
\begin{eqnarray}\label{eq:t19}
\sup_{t>0} \int_{\Omega}g(U)\phi_1\le \frac{1+\mu_1+\nu}{\lambda}M_1.
\end{eqnarray}
Now let us show that both $U$ and $V$ are bounded in $L_1$ uniformly in time.
Let $\zeta$ be the solution of \eqref{torsion1d}.
Observe that estimates \eqref{eq:t15}, \eqref{eq:t17} and \eqref{eq:t19}
imply that
\begin{eqnarray}\label{eq:t110}
\int_{\Omega} U\zeta, ~ \int_{\Omega} V\zeta,~
\sup_{t>0} \int_{\Omega}g(U)\zeta \le  M_3
\end{eqnarray}
for some constant $ M_3$  independent of time.
Next setting in \eqref{eq:t11} $\phi=\psi=\zeta$ and integrating the result on $(t,t+1)$ we have
\begin{multline}\label{eq:t111}
\int_{\Omega} U \le \int_t^{t+1} \int_{\Omega} U\\
=\int_{\Omega} U(t,\cdot)\zeta-\int_{\Omega} U(t+1,\cdot)\zeta+
\lambda \int_t^{t+1}\int_{\Omega} g(U)\zeta+\nu\int_t^{t+1} (V-U)\zeta\le M_4, \\
d\int_{\Omega} V \le d \int_t^{t+1} \int_{\Omega}V =\alpha \left(\int_{\Omega} V(t,\cdot)\zeta-\int_{\Omega} V(t+1,\cdot)\zeta\right)+\nu\int_t^{t+1} (U-V)\zeta \le M_5.
\end{multline}
The first and the last  inequalities in both equations of  \eqref{eq:t111} hold since $U$ and $V$ are non-decreasing functions of time (lemma \ref{l:1}) and by \eqref{eq:t110} respectively.
Thus,
\begin{eqnarray}\label{eq:t112}
\sup_{t>0} ||U(t)||_{L^1(\Omega)}\le M_4, \quad \sup_{t>0} ||V(t)||_{L^1(\Omega)}\le M_5.
\end{eqnarray}

From \eqref{eq:t112} and monotone convergence theorem we deduce that $U$ and $V$ have
a limit $u$, $v$ in $L^1(\Omega)$. Moreover by \eqref{eq:t18} we have that $g(U)$ converges
to $g(u)$ in $L^1(\Omega, \delta(x)dx)$ as $t\to \infty$.

Integrating \eqref{eq:t11} on $(t,t+1)$ we have
\begin{eqnarray}\label{eq:t113}
&& \int_{\Omega} U\phi \big|_t^{t+1}+\int_{\Omega} U(-\Delta \phi)=
\lambda \int_{\Omega}  g(U)\phi +\nu \int_{\Omega} (V-U)\phi, \nonumber \\
&&\alpha  \int_{\Omega} V\psi\big|_t^{t+1}+d\int_{\Omega} V(-\Delta \psi)=
\nu\int_{\Omega} (U-V) \psi.
\end{eqnarray}
Finally letting $t\to \infty$ we have
\begin{eqnarray}\label{eq:t114}
&& \int_{\Omega} u(-\Delta \phi)=
\lambda \int_{\Omega}  g(u)\phi +\nu \int_{\Omega} (v-u)\phi, \nonumber \\
&&d\int_{\Omega} v(-\Delta \psi)=
\nu\int_{\Omega} (u-v) \psi.
\end{eqnarray}
Therefore, the $L^1$-limit of $U$ and $V$ as $t\to\infty$ is a weak solution of \eqref{eq:P},
as defined in \eqref{eq:weak_sol}.
Let us also note that $U$ and $V$ are non-decreasing in time.
In view of the parabolic comparison principle for \eqref{eq:PT} we conclude that
the limit of $U$ and $V$ is a minimal weak solution of elliptic problem \eqref{eq:P}.
\end{proof}

\bigskip

{\bf Acknowledgments.} The work of PVG was supported, in a part, by the Faculty Research
Committee of The University of Akron via grant FRG 1774 , the United States-Israel Binational Science Foundation via grant 2012057 and NSF via grant DMS-1119724.
Part of this research was carried out while VM was visiting The University of Akron, he
thanks the Department of Mathematics for its support and hospitality.
The authors are grateful to C.B. Muratov for valuable discussions.

\end{document}